\documentclass[12pt]{amsart}

\usepackage{amsmath}
\usepackage{amsfonts}
\usepackage{amsthm}

\usepackage[cmtip,arrow]{xy}
\usepackage{pb-diagram, pb-xy}

\newtheorem{theorem}{Theorem}

\newtheorem{lemma}{Lemma}
\newtheorem{proposition}{Proposition}

\begin{document}
\title{Ding injective envelopes in the category of complexes}
%\author{Alina Iacob}
%\address[A. Iacob]{Department of Mathematical Sciences. Georgia Southern University. Statesboro (GA) 30460-8093. USA}
%\email{aiacob@GeorgiaSouthern.edu}
\author{James Gillespie}
\address{J.G. \ Ramapo College of New Jersey \\
         School of Theoretical and Applied Science \\
         505 Ramapo Valley Road \\
         Mahwah, NJ 07430\\ U.S.A.}
\email[Jim Gillespie]{jgillesp@ramapo.edu}
\urladdr{http://pages.ramapo.edu/~jgillesp/}

\author{Alina Iacob}
\address{A.I. \ Department of Mathematical Sciences \\
         Georgia Southern University \\
         Statesboro (GA) 30460-8093 \\ U.S.A.}
\email[Alina Iacob]{aiacob@GeorgiaSouthern.edu}
\urladdr{https://sites.google.com/a/georgiasouthern.edu/aiacob/home}

\begin{abstract}
A complex $X$ is called Ding injective if there exists an exact sequence of injective complexes $\ldots \rightarrow E_1 \rightarrow E_0 \rightarrow E_{-1} \rightarrow \ldots$ such that $X = Ker(E_0 \rightarrow E_{-1})$, and the sequence remains exact when the functor $Hom(A,-)$ is applied to it, for any $FP$-injective complex $A$.
We prove that, over any ring $R$, a complex is Ding injective if and only if it is a complex of Ding injective modules. We use this to show that the class of Ding injective complexes is enveloping over any ring.
\end{abstract}
\maketitle

%The existence of Gorenstein projective precovers is still one of the main open questions in Gorenstein homological algebra. Their existence is known over ...\\

%We give a sufficient condition for the class of Gorenstein projectives be precovering. We show that (over any ring $R$) if every Gorenstein projective module is Gorenstein flat then the class of Gorenstein projective modules is special precovering. More precisely, we show that in this case, the class of Gorenstein projective modules coincides with that of projectively gorenstein flat modules, and therefore, by (St. - Saroch) it is special precovering in $R-Mod$.

%\section{preliminaries}

%\section{main result}
\section{introduction}

It is an important question to establish relationships between a complex $X$ and the $R$-modules
$X_n$, $n \in Z$, for a given ring $R$. For example, it is well known that a complex $I$ is injective (projective, flat respectively) if and only if it is an exact complex and all its cycles are injective (projective, flat respectively) modules. We consider here the relationship between the Ding injectivity of a complex $X$ and the Ding injectivity of the modules $X_n$, $n \in Z$.

The Ding injective modules were introduced by Ding and Mao, %in 2008,
 in \cite{ding:08:ding.inj}, where they were called Gorenstein $FP$-injective modules. Later, Gillespie renamed these modules, calling them Ding injective modules (in \cite{gil:17:ding}). The definition uses $FP$-injective modules, so we recall first that a module $A$ is called $FP$-injective (or absolutely pure) if $Ext^1(F,A)=0$ for all finitely presented $R$-modules $F$. We use $\mathcal{FI}$ to denote this class of modules. The Ding injective modules are the cycles of the exact complexes of injective modules that remain exact when applying a functor $Hom(A, -)$, with $A$ any $FP$-injective module.\\

The Ding injective complexes are defined in a similar manner: a complex $X$ is Ding injective if it is a cycle of an exact complex of injective complexes $E = \ldots \rightarrow E_1 \rightarrow E_0 \rightarrow E_{-1} \rightarrow \ldots$ such that $Hom(A,E)$ is exact for any $FP$-injective complex $A$. This $Hom$ too is just the abelian group of chain maps from $A$ to $E$. \\

This class of complexes has been studied before - see for example \cite{YLL13}, \cite{gil:17:ding}, and \cite{Geng} . In \cite{YLL13} (Theorem 3.20),  the authors prove that a complex $X$ is Ding injective if and only if each term $X_i$ is Ding injective and any chain map $A \rightarrow X$ from any $FP$-injective complex $A$ is null homotopic (i.e. the Hom complex $\mathcal{H}om (A,X)$ is exact, for any $FP$-injective complex $A$). In \cite{gil:17:ding}, Corollary 4.4, it is  proved that over a Ding-Chen ring (i.e. a two sided coherent ring having finite $FP$-injective dimension on both sides), the Ding injective complexes are precisely the complexes of Ding injective modules. In~\cite{YE20} (Theorem~4.1) this result was extended to all coherent rings. Here we prove that this result holds over any ring: a complex is Ding injective if and only if it is a complex of Ding injective modules.

%We also recall that, by [Wang, and Z.K. Liu], Theorem 2.10, a complex $A$ is $FP$-injective if and only if it is exact and $Z_nA$ is an $FP$-injective module for all $n \in Z$.

%We recall also that the Gorenstein injective modules (introduced by Enochs and Jenda in 1995, in [4]) are the cycles of exact complexes of injective modules that remain exact when applying a functor $Hom(A,−)$ with $A$ any injective module. So, in particular, every Ding injective module is a Gorenstein injective module. We denote by $\mathcal{DI}$ the class of Ding injective modules, and by $\mathcal{GI}$ that of Gorenstein injective modules.

It was recently proved (\cite{gil:iac}, Theorem 5.4) that over any ring $R$, the class of Ding injective modules, $\mathcal{DI}$, is the right half of a perfect cotorsion pair in $R-Mod$. Consequently, the class of Ding injective modules is enveloping over any ring. We show that a similar result holds for the class of Ding injective complexes: over any ring $R$, the class of Ding injective complexes is an enveloping class in $Ch(R)$, the category of chain complexes of $R$-modules.

We prove first (Theorem 2) that, over any ring $R$, a complex is Ding injective if and only if it is a complex of Ding injective modules. Then we show (Proposition 3) that the class of Ding injective complexes is special preenveloping in $Ch(R)$.
%the right half of a complete and hereditary cotorsion pair. This implies that every complex has a special Ding injective preenvelope.
Theorem 3 shows that in fact the class of Ding injective complexes is enveloping in $Ch(R)$, for any ring $R$.
\section{preliminaries}\label{sec-preliminaries}

%- We show (Proposition 3) that if $R$ is a coherent ring then $(^\bot \mathcal{DI}, \mathcal{DI})$ is a hereditary cotorsion pair.
%We show that a similar result holds for the class of Ding injective complexes: it is an enveloping class in $Ch(R)$, for any ring $R$.\\

%We start by proving that over any ring $R$, a complex is Ding injective if and only if it is a complex of Ding injective modules.

In this paper, $R$ denotes an associative ring with unity, $R-Mod$ denotes the category of left
$R$-modules and $Ch (R)$ denotes the abelian category of complexes of left R-modules.
A complex $$\ldots \rightarrow C_1 \xrightarrow{d_1} C_0 \xrightarrow{d_0} C_{-1} \rightarrow \ldots$$ will be denoted by $(C, d)$, or simply by $C$.
The $n$th cycle of $C$, $Ker(d_n)$, is denoted $Z_n(C)$. The $n$th boundary of $C$, $Im(d_n)$, is denoted $B_n(C)$. The $n$th homology module of $C$ is $H_n(C) = Z_n(C)/B_{n+1}(C)$. The complex $C$ is exact (or acyclic) if $H_n(C)=0$ for all $n$.\\

For two complexes $C, D \in Ch(R)$, we let $Hom(C, D)$ denote the abelian group of chain maps from $C$ to $D$ in $Ch(R)$. We use the notation $Ext^i
(C, D)$ where $i \ge 1$ for the groups that arise from the right derived functor of $Hom$.\\

We use the notation suggested by Brown in \cite{brown} and denote by $\mathcal{H}om (C,D)$ the usual complex formed from two complexes $C$ and $D$. Then $Z_0 \mathcal{H}om (C,D)$ is the group $Hom(C,D)$ of morphisms from $C$ to $D$.
%Definition $\mathcal{H}om$ \\

%Throughout, $R$ denotes an associative ring with unity. Unless otherwise specified, by $R$-module we mean left $R$-module.

%We recall first that a module $M$ is said to be Ding injective if it is a cycle of an exact complex of injective modules that stays exact when applying the functor $Hom(A,-)$ for any $FP$-injective module $A$.
As already mentioned, the Ding injective complexes are  defined in a similar manner with the Ding injective modules. The definition uses $FP$-injective complexes, so we recall that a complex $C$ is $FP$-injective  if $Ext^1(F, C) = 0$ for every finitely presented complex $F$.\\ A complex $X$ is called \emph{Ding injective} if it is a cycle of an exact complex of injective complexes $E = \ldots \rightarrow E_1 \rightarrow E_0 \rightarrow E_{-1} \rightarrow \ldots$ such that $Hom(A,E)$ is exact for any $FP$-injective complex $A$.\\
It is known (\cite{YLL13}, Theorem 3.20), that a complex $X$ is Ding injective if and only if each term $X_i$ is Ding injective and any chain map $A \rightarrow X$ from any $FP$-injective complex $A$ is null homotopic (or equivalently, the Hom complex $\mathcal{H}om (A,X)$ is exact, for any $FP$-injective complex $A$).

%We also recall that, by [Wang, and Z.K. Liu], Theorem 2.10, a complex $A$ is $FP$-injective if and only if it is exact and $Z_nA$ is an $FP$-injective module for all $n \in Z$.

We recall that a complex $M$ has a Ding injective \emph{preenvelope} if there exists a chain map $l: M \rightarrow A$ with $A$ a Ding injective complex and
such that for any Ding injective complex $A'$, any chain map
$h: M \rightarrow A'$ factors through $l$ ($ h = v l$ for some $v
\in Hom(A,A')$).
\[
\begin{diagram}
\node{M}\arrow{s,l}{h}\arrow{e,t}{l}\node{A}\arrow{sw,b,..}{v}\\
\node{A'}
\end{diagram}
\]

Such a preenvelope $l$ is said to be an \emph{envelope} if
it has one more property: any $v \in Hom(A,A)$ such that $vl=l$ is
an automorphism of $A$.\\
Any Ding injective preenvelope $l : M \rightarrow A$ is called \emph{special} if we have $Ext^1(cok(i),X)=0$ for all Ding injective complexes $X$. %its cokernel is in the left orthogonal class of that of Ding injective complexes.

One way to prove that a class $\mathcal{L}$ is special preenveloping is showing that $\mathcal{L}$ is the right half of a complete cotorsion pair in $Ch(R)$.
%De pus def mai general given a class of objects in abelian cat (de vazut cum pun ei)

Given a class of objects $\mathcal{C}$ in a Grothendieck category $\mathcal{A}$, we denote by $\mathcal{C}^\bot$ its right orthogonal class, i.e. the class of objects $X$ such that $Ext^1(C,X)=0$ for any $C \in \mathcal{C}$. The left orthogonal class of $\mathcal{C}$ is defined dually. We recall that a pair of classes of objects $(\mathcal{C}, \mathcal{L})$, is a \emph{cotorsion pair} if $\mathcal{C} ^ \bot = \mathcal{L}$ and $^ \bot \mathcal{L} = \mathcal{C}$. A cotorsion pair is \emph{complete} if for any object $M$ there are exact sequences $0 \rightarrow L \rightarrow C \rightarrow M \rightarrow 0$ and respectively $0 \rightarrow M \rightarrow L' \rightarrow C' \rightarrow 0$ with $C, C' \in \mathcal{C}$ and $L, L' \in \mathcal{L}$. We also recall that a cotorsion pair is said to be \emph{hereditary} if $Ext^i(C,L)=0$ for all $i \ge 1$, all $C \in \mathcal{C}$, all $L \in \mathcal{L}$. It is known that this is equivalent with the class $\mathcal{C}$ being closed under kernels of epimorphisms, and it is also equivalent with the condition that $\mathcal{L}$ is closed under cokernels of monomorphisms.\\

%Def precovers/covers, preenvelopes/envelopes si connection cu complete pairs.\\

A cotorsion pair $(\mathcal{C}, \mathcal{L})$ is said to be \emph{perfect } if $\mathcal{C}$ is covering and $\mathcal{L}$ is enveloping. \\

Given a cotorison pair $(\mathcal{A}, \mathcal{B})$ in the category of $R$-modules, Gillespie introduced (\cite{Gill04}) four classes of complexes in $Ch(R)$ that are associated with it: \\
1. An acyclic complex $X$ is an {$\mathcal{A}$-complex} if $Z_j(X) \in \mathcal{A}$ for all integers $j$. We denote by $\widetilde{\mathcal{A}}$ the class of all acyclic $\mathcal A$-complexes.\\
2. An acyclic complex $U$ is a $\mathcal{B}$-complex if $Z_j(X) \in \mathcal{B}$ for all integers $j$. We denote by $\widetilde{\mathcal{B}}$ the class of all acyclic $\mathcal B$-complexes.\\
3. A complex $Y$ is a {dg-$\mathcal{A}$ complex} if each $Y_n \in \mathcal{A}$ and each map $Y\to U$ is null-homotopic, for each complex $U\in \widetilde{\mathcal{B}}$. We denote by $ dg (\mathcal{A})$ the class of all dg-$\mathcal{A}$ complexes. \\
4. A complex $W$ is a dg-$\mathcal{B}$ complex if each $W_n \in \mathcal{B}$ and each map $V\to W$ is null-homotopic, for each complex $V\in \widetilde{\mathcal{A}}$. We denote by $ dg (\mathcal{B})$ the class of all dg-$\mathcal{B}$ complexes.\\

%We refer the reader to the preliminaries section for the undefined notions.\\
Yang and Liu showed in \cite{yang-liu11}, Theorem 3.5, that when $(\mathcal{A}, \mathcal{B})$ is a complete hereditary cotorsion pair in $R$-Mod, the pairs $(dg (\mathcal{A}), \widetilde{\mathcal{B}})$ and $(\widetilde{\mathcal{A}}, dg (\mathcal{B}))$ are complete and hereditary cotorsion pairs in the category of complexes $Ch(R)$. Moreover, by Gillespie \cite{Gill04}, we have that  $\widetilde{\mathcal{A}}=dg (\mathcal{A})\bigcap \mathcal E$ and $ \widetilde{\mathcal{B}}=dg (\mathcal{B})\bigcap \mathcal E$ (where $\mathcal E$ is the class of all acyclic complexes). For example, from the (complete and hereditary) cotorsion pairs $(Proj,R\textrm{-Mod})$ and $(R\textrm{-Mod},Inj)$, one obtains the standard (complete and hereditary) cotorsion pairs $(\mathcal E,dg(Inj))$ and $(dg(Proj),\mathcal E)$.\\

We use $\mathcal{DI}$ to denote the class of Ding injective modules. It was recently proved (\cite{gil:iac}, Theorem~5.4), that $(^\bot \mathcal{DI}, \mathcal{DI})$ is a complete hereditary cotorsion pair over any ring $R$ (in fact, this is a perfect cotorsion pair). Then by \cite{yang-liu11}, Theorem 3.5, $(\widetilde{^\bot \mathcal{DI}}, dg (\mathcal{DI}))$ is a complete hereditary cotorsion pair in the category of complexes, $Ch(R)$.

We recall below some of the results that we use.\\
\begin{theorem}(this is part of \cite{wl}, Theorem 2.10)
Let $C$ be a complex. Then the following statements are equivalent.\\
(1) $C$ is $FP$-injective.\\
(2) $C$ is exact and $Z_n(C)$ is $FP$-injective in $R-Mod$ for all $n \in Z$.
\end{theorem}

We also recall that a short exact sequence of modules $0 \rightarrow M'\rightarrow M  \rightarrow M'' \rightarrow 0$ is \emph{pure exact} if and only if the induced sequence of abelian groups $0 \rightarrow Hom(Y, M')\rightarrow Hom(Y, M) \rightarrow Hom(Y, M'') \rightarrow 0$ is exact for any finitely presented module Y.\\
By definition, a complex $F$ is \emph{pure acyclic} if it is acyclic and such that the short exact sequences of modules $0 \rightarrow Z_n(F) \rightarrow F_n \rightarrow Z_{n-1}(F) \rightarrow 0$ are pure for all $n$.\\

The following characterization of pure acyclic complexes is part of \cite{E}, Proposition 2.2.
\begin{proposition}(\cite{E}, Proposition 2.2)
The following are equivalent for a complex $C$.\\
(1) $C$ is a pure acyclic complex.\\
(2) $Hom(Y, C)$ is acyclic for any finitely presented module $Y$.
\end{proposition}

Theorem 1 and Proposition 1 imply that every $FP$-injective complex is pure acyclic.

By \cite{Sto}, Definition 6.7, a complex $X$ is called \emph{coacyclic} if $Ext^1(X,I) =0$ for every complex of injective modules $I$.\\
%As in \cite{Sto}, we use $C_{pac}(\mathcal{FI})$ to denote the class of pure acyclic complexes of $FP$-injective objects. $C_{coac}(\mathcal{G})$ are the coacyclic complexes, and $C(\mathcal{FI})$ are the complexes of $FP$-injective objects.

\begin{lemma}(\cite{Sto}, Lemma 6.10)
Let $\mathcal{G}$ be a locally finitely presentable Grothendieck category. Then $C_{pac}(\mathcal{FI}) \subseteq C_{coac}(\mathcal{G}) \bigcap C(\mathcal{FI})$, where $C_{pac}(\mathcal{FI})$ is the class of pure acyclic complexes of $FP$-injective objects, $C_{coac}(\mathcal{G})$ are the coacyclic complexes, and $C(\mathcal{FI})$ are the complexes of $FP$-injective objects.
\end{lemma}

By \cite{christensen}, Corollary 4.6, the category $Ch(R)$ is locally finitely presented. So we obtain:\\
\begin{lemma}
Every pure acyclic complex of $FP$-injective modules is coacyclic.
\end{lemma}

We also recall that by \cite{gil:gor}, Definition 3.4, a complete cotorsion pair $(\mathcal{W}, \mathcal{F})$ is called an \emph{injective cotorsion pair} if the class $\mathcal{W}$ is thick and if $\mathcal{W} \bigcap \mathcal{F}$ coincides with the class of injective objects. \\

The following is \cite{gil:gor}, Proposition 3.6:
\begin{proposition}(Characterizations of injective cotorsion pairs).
Suppose $(\mathcal{W}, \mathcal{F})$ is a complete cotorsion pair in an abelian category $\mathcal{A}$ with enough injectives. Then
each of the following statements are equivalent:\\
(1) $(\mathcal{W}, \mathcal{F})$ is an injective cotorsion pair.\\
(2) $(\mathcal{W}, \mathcal{F})$ is hereditary and $\mathcal{W} \bigcap \mathcal{F}$ equals the class of injective objects.
(3) $\mathcal{W}$ is thick and contains the injective objects.

\end{proposition}

\section{results}
We start by proving that over any ring $R$, a complex is Ding injective if and only if it is a complex of Ding injective modules.
\begin{theorem}
Let $R$ be any ring. A complex $X$ is Ding injective if and only if it is a complex of Ding injective $R$-modules.
\end{theorem}

\begin{proof}
The condition is necessary by \cite{YLL13} Theorem 3.20.\\
We show that it is also a sufficient condition. Let $X$ be a complex of Ding injective modules. Since $(\widetilde{^\bot \mathcal{DI}}, dg (\mathcal{DI}))$ is a complete cotorsion pair in $Ch(R)$, there is an exact sequence $0 \rightarrow J \rightarrow I \xrightarrow{\pi} X \rightarrow 0$ with $I \in \widetilde{^\bot \mathcal{DI}}$ and with $J \in dg (\mathcal{DI}))$. This gives a short exact sequence of modules, $0 \rightarrow J_n \rightarrow I_n \rightarrow X_n \rightarrow 0$, for any $n \in Z$. Since both $J_n$ and $X_n$ are Ding injective modules, it follows that for each $n$ we have $I_n \in {}^\bot \mathcal{DI} \bigcap \mathcal{DI} = Inj$. So $I$ is an exact complex of injective $R$-modules.\\
Let $C$ be an $FP$-injective complex, and let $f:C \rightarrow X$ be a chain map. Since for each $n$, $Z_n (C)$ is $FP$-injective (by Theorem 1), so in $^\bot \mathcal{DI}$, it follows that $C \in \widetilde{^\bot \mathcal{DI}}$, and therefore $Ext^1(C,J)=0$. Thus $Hom(C,I) \rightarrow Hom(C,X) \rightarrow 0$ is exact. It follows that $f$ factors through $\pi$, $f = \pi g$, for some $g: C \rightarrow I$. So it suffices to show that $g$ is null homotopic.\\

%diagrama\\

\[
\begin{diagram}
\node{}\node{C}\arrow{sw,t,..}{g}\arrow{s,r}{f}\\
\node{I}\arrow{e,t}{\pi}\node{X}
\end{diagram}
\]

By Theorem 1 and Proposition 1, %\cite{E}, Proposition 2.2,
the complex $C$ is pure acyclic. By Lemma 2, %\cite{Sto}, Lemma 6.10,
any pure acyclic complex of $FP$-injective modules, in particular $C$, is coacyclic. So  %, i.e. $Ext^1(C, K)=0$ for any complex of injective modules $K$. In particular,
$Ext^1(C,I)=0$. By \cite{enochs:96:orthogonality}, Corollary 3.3, this is equivalent to $\mathcal{H}om (C,I)$ being exact. Thus any map from $C$ to $I$, in particular $g$, is homotopic to zero. It follows that $f= \pi g$ is null homotopic. So $X$ is a Ding injective complex by~\cite[Theorem~3.20]{YLL13}.
\end{proof}

We can prove now that the class of Ding injective complexes is the right half of a complete cotorsion pair in $Ch(R)$ (for any ring $R$).\\
We use $dw(\mathcal{DI})$ to denote the class of complexes of Ding injective modules. By Theorem 2, these are precisely the Ding injective complexes.\\

\begin{proposition}
Let $R$ be any ring. The class of Ding injective complexes is special preenveloping in $Ch(R)$.
\end{proposition}

\begin{proof}
Since $(^\bot \mathcal{DI}, \mathcal{DI})$ is an injective cotorsion pair cogenerated by a set (\cite{gil:iac}, Lemma 3.7, Theorem 3.11,  and Theorem 5.4), it follows that $(^\bot dw (\mathcal{DI}), dw(\mathcal{DI}))$ is a complete cotorsion pair in $Ch(R)$ (by \cite{gil:gor}, Proposition 7.2(1). So $dw(\mathcal{DI}))$ is special preenveloping. By Theorem 2 above, $dw(\mathcal{DI}))$ is the class of Ding injective complexes.
\end{proof}

%We use $dw(\mathcal{DI})$ to denote the class of complexes of Ding injective modules. By Theorem 2, these are precisely the Ding injective complexes.\\
We show that, in fact, the class of Ding injective complexes, $dw (\mathcal{DI})$,  is enveloping in $Ch(R)$, for any ring $R$.\\

We show first that $^\bot dw(\mathcal{DI})$ is a covering class in $Ch(R)$. By the proof of Proposition 3, $^\bot dw(\mathcal{DI})$ is a special precovering class. It is known that a precovering class that is also closed under direct limits is covering. So it suffices to show that $^\bot dw(\mathcal{DI})$ is closed under direct limits. The proof uses the following remarks:\\

\begin{lemma}
$(^\bot dw (\mathcal{DI}, dw(\mathcal{DI}))$ is a herediatry cotorsion pair.
\end{lemma}
\begin{proof}
Let $0 \rightarrow A' \rightarrow A \rightarrow A" \rightarrow 0$ be an exact sequence with both $A'$ and $A\in dw(\mathcal{DI})$. Then for each $n$ we have an exact sequence of modules $0 \rightarrow A_n' \rightarrow A_n \rightarrow A_n" \rightarrow 0$, with $A_n', A_n" \in \mathcal{DI}$. Since $(^\bot \mathcal{DI}, \mathcal{DI})$ is a hereditary cotorsion pair, it follows that $A_n" \in \mathcal{DI}$. Thus $A" \in dw(\mathcal{DI})$.
\end{proof}

\begin{lemma}
$^\bot dw(\mathcal{DI})$ is a thick class.
\end{lemma}

\begin{proof}
As a left orthogonal class, $^\bot dw(\mathcal{DI})$ is closed under extensions, and by Lemma 3, it is also closed under kernels of epimorphisms. \\We check that $^\bot dw(\mathcal{DI})$ is closed under cokernels of monomorphisms. Let $0 \rightarrow D' \rightarrow D \rightarrow D" \rightarrow 0$ be a short exact sequence with both $D'$ and $D$ in $^\bot dw(\mathcal{DI})$. Let $A \in dw(\mathcal{DI})$. The long exact sequence $0 = Ext^1(D',A) \rightarrow Ext^2(D",A) \rightarrow Ext^2(D,A) =0$ (by Lemma 3) gives that $Ext^2(D",A)=0$ for all Ding injective complexes $A$. By the definition of the Ding injective complexes, there is an exact sequence $0 \rightarrow A' \rightarrow I \rightarrow A \rightarrow 0$ with $A'$ a Ding injective complex and with $I$ an injective complex. By the above, $Ext^2(D', A')=0$. The associated exact sequence $ 0 = Ext^1(D",I) \rightarrow Ext^1(D", A) \rightarrow Ext^2(D", A')=0$ gives that $Ext^1(D",A)=0$ for all $A \in dw(\mathcal{DI})$. Thus $D" \in {}^\bot dw(\mathcal{DI})$.\\
- As a left orthogonal class, $^\bot dw(\mathcal{DI})$  is closed under direct summands.
\end{proof}

\begin{proposition}
The class $^\bot dw(\mathcal{DI})$ is covering in $Ch(R)$
\end{proposition}
\begin{proof}
Since $(^\bot dw(\mathcal{DI}), dw(\mathcal{DI}))$ is a cotorsion pair with $^\bot dw(\mathcal{DI})$ thick, it follows that  the class $^\bot dw(\mathcal{DI})$ is closed under direct limits (by \cite{gil:17:ding}, Proposition 3.2). Thus $^\bot dw(\mathcal{DI})$ is a precovering class that is also closed under direct limits. So $^\bot dw(\mathcal{DI})$ is a covering class.
\end{proof}

We recall that a cotorsion pair $(\mathcal{C}, \mathcal{L})$ is called \emph{perfect} if $\mathcal{C}$ is covering and $\mathcal{L}$ is enveloping.\\

In order to prove that the class of Ding injective complexes is enveloping in $Ch(R)$, for any ring $R$, we use the following result (showing that $(^\bot dw(\mathcal{DI}), dw(\mathcal{DI}))$  is a perfect cotorsion pair if and only if $^\bot dw(\mathcal{DI})$ is covering).\\
\begin{proposition}
Let $R$ be any ring. The following are equivalent:\\
1. The class of Ding injective complexes is enveloping.\\
2. $(^\bot dw(\mathcal{DI}), dw(\mathcal{DI}))$ is a perfect cotorsion pair.\\
3. The left orthogonal class of that of Ding injective complexes, $^\bot dw(\mathcal{DI})$, is covering.\\
\end{proposition}

\begin{proof}
1 $\Rightarrow$ 2. By Proposition 4, every Ding injective complex has a $^\bot dw(\mathcal{DI})$ cover.\\ %By Lemma 3,  $(^\bot dw(\mathcal{DI}), dw(\mathcal{DI}))$ is a hereditary cotorsion pair. By \cite{enochs:04:gorflat.covers}Theorem 1.4, this cotorsion pair is perfect if and only if $dw(\mathcal{DI})$ is enveloping and everyDing injective complex has a $^\bot dw(\mathcal{DI})$ cover. By Proposition 4, every Ding injective complex has a $^\bot dw(\mathcal{DI})$ cover.\\

3 $\Rightarrow$ 2. By Lemma 3,  $(^\bot dw(\mathcal{DI}), dw(\mathcal{DI}))$ is a hereditary cotorsion pair. By \cite{enochs:04:gorflat.covers}, Theorem 1.4, this cotorsion pair $(^\bot dw(\mathcal{DI}), dw(\mathcal{DI}))$ is perfect if and only if $^\bot dw(\mathcal{DI})$ is covering and every $X \in {}^\bot dw(\mathcal{DI})$ has a Ding injective envelope.\\
Let $X \in {}^\bot dw(\mathcal{DI})$. Consider the exact sequence $0 \rightarrow X \rightarrow I \rightarrow Y \rightarrow 0$ with $X \rightarrow I$ the injective envelope of $X$. Since both $X$ and $I$ are in $^\bot dw(\mathcal{DI})$, it follows (by Lemma 4) that $Y \in {}^\bot dw(\mathcal{DI})$. So the sequence is $Hom(-, dw(\mathcal{DI}))$ exact. Thus $X \rightarrow I$ is a special Ding injective preenvelope of $X$. Since any $u : I\xrightarrow{} I$ that is the identity on $X$ is an automorphism of $I$, it follows that $X \rightarrow I$ is a Ding injective envelope.

2 $\Rightarrow$ 1 and 2 $\Rightarrow$ 3 follow from the definition of a perfect pair.
\end{proof}

\begin{theorem}
Let $R$ be any ring. The class of Ding injective complexes is enveloping in $Ch(R)$.
\end{theorem}
\begin{proof}
By Proposition 5 above the class of Ding injective complexes is enveloping if and only if $^\bot dw(\mathcal{DI})$ is covering. By Proposition 4, the class $^\bot dw(\mathcal{DI})$ is covering.
\end{proof}

\end{document}